\documentclass[12pt]{article}

\usepackage{amsfonts}
\usepackage{amsmath}
\usepackage{amscd}
\usepackage{amssymb}
\usepackage{amsthm}
\usepackage{amsxtra}
\usepackage{diagrams}

\newcommand{\bbC}    {\mathbb C}

\newcommand{\bbR}    {\mathbb R}
\newcommand{\bbF}     {\mathbb F}

\newcommand{\bbZ}    {\mathbb Z}
\newcommand{\bbN}    {\mathbb N}

\newcommand{\cA}    {{\cal A}}
\newcommand{\cB}    {{\cal B}}
\newcommand{\cC}    {{\cal C}}
\newcommand{\cD}    {{\cal D}}
\newcommand{\cE}    {{\cal E}}

\newcommand{\cI}    {{\cal I}}

\newcommand{\cK}    {{\cal K}}
\newcommand{\cL}    {{\cal L}}
\newcommand{\cM}    {{\cal M}}
\newcommand{\cN}     {{\cal N}}

\newcommand{\cS}    {{\cal S}}

\newcommand{\cU}    {{\cal U}}
\newcommand{\cV}    {{\cal V}}

\newcommand{\al}    {\alpha}
\newcommand{\be}    {\beta}
\newcommand{\de}    {\delta}
\newcommand{\De}    {\Delta}
\newcommand{\eps}    {\epsilon}

\newcommand{\ka}    {\kappa}

\newcommand{\om}    {\omega}
\newcommand{\Ga}    {\Gamma}

\newcommand{\Om}    {\Omega}
\newcommand{\si}     {\sigma}
\newcommand{\vk}     {\varkappa}
\newcommand{\vfi}    {\varphi} 
\newcommand{\vte}    {\vartheta}

\newcommand{\A}      {\mathrm A}
\newcommand{\B}       {\mathrm B}

\newcommand{\E}       {\mathrm E}

\newcommand{\I}         {\mathrm I}

\newcommand{\X}       {\mathrm X}

\newcommand{\U}        {\mathrm U}

\newcommand{\Z}         {\mathrm Z}

\newcommand{\Sr}        {\mathrm S}
\newcommand{\alg}       {\mathrm{alg}}

\newcommand{\Lie}       {\mathrm{Lie}}

\newcommand{\bA}     {\mathbf A}
\newcommand{\ba}     {\mathbf a}

\newcommand{\bD}      {\mathbf D}

\newcommand{\bm}    {\mathbf m}
\newcommand{\bU}     {\mathbf U}

\newcommand{\bX}    {\mathbf X}

\newcommand{\bR}     {\mathbf R}

\newcommand{\bu}    {\mathbf u}

\newcommand{\bad}     {\mathbf{ad}}

\newcommand{\fA}    {\mathfrak A}

\newcommand{\fD}    {\mathfrak D}

\newcommand{\fF}      {\mathfrak F}

\newcommand{\fM}    {\mathfrak M}

\newcommand{\w}       {\wedge}
\newcommand{\ck}      {\check} 
\newcommand{\na}      {\nabla} 
\newcommand{\bcA}    {\pmb{\cal A}}
\newcommand{\bcD}    {\pmb{\cal D}}
\newcommand{\bcE}    {\pmb{\cal E}}

\newcommand{\bGa}   {\pmb{\Gamma}}
\newcommand{\bfD}    {\pmb{\mathfrak D}}
\newcommand{\bfM}    {\pmb{\mathfrak M}}

\newcommand{\bna}    {\pmb{\nabla}}
\newcommand{\bphi}   {\pmb{\phi}}

\newcommand{\brho}   {\pmb{\rho}}

\newcommand{\bvte}   {\pmb{\vartheta}}

\newcommand{\p}        {\partial}
\newcommand{\bp}       {\pmb{\partial}}

\DeclareMathOperator{\id}    {\,id}
\DeclareMathOperator{\Hom}   {Hom}
\DeclareMathOperator{\End}   {End}

\DeclareMathOperator{\im}    {Im}
\DeclareMathOperator{\ke}    {Ker}
\DeclareMathOperator{\com}   {{\scriptstyle\circ}\,}
\DeclareMathOperator{\ad}    {ad}

\DeclareMathOperator{\Sym}     {Sym}

\newtheorem{theorem}{Theorem}
\newtheorem{lemma}{Lemma}
\newtheorem{prop}{Proposition}
\newtheorem{cor}{Corollary}
\theoremstyle{definition} 
\newtheorem{rem}{Remark} 
\newtheorem{exa}{Example}
\newtheorem{defi}{Definition}
\newtheorem{ass}{Assumption}

\begin{document}

\title{On analysis in differential algebras and modules}
\author
          {
             Zharinov V.V.
             \thanks{Steklov Mathematical Institute}
             \thanks{E-mail: zharinov@mi.ras.ru}
             \thanks{This research was performed at the Steklov Mathematical Institute 
             of Russian Academy of Sciences and supported by a grant from the Russian 
             Science Foundation (Project No. 14-50-00005).}
            }
\date{}
\maketitle

\begin{abstract}
A short introduction to the mathematical methods and technics 
of differential algebras and modules adapted to the problems of mathematical 
and theoretical physics is presented.  
\end{abstract} 


\section{Introduction.} 

Differential algebras (see, for example, \cite{RJF},\cite{KER},\cite{KI}) 
are widely known and used in algebra and topology, 
while their applications in mathematical physics are far less acknowleged 
(but used implicitly, especially in partial differential equations). 
Here we propose a short introduction to the mathematical methods and technics 
of differential algebras and modules adapted to the problems of mathematical and theoretical 
physics. Exposition is based on the personal experience of the author, 
the books \cite{RJF},\cite{KER},\cite{KI},\cite{SM},\cite{RG},\cite{PO},\cite{Z6} 
and the works \cite{AV},\cite{TT},\cite{AI},\cite{Z0},\cite{Z1},\cite{Z2}.

We freely use notation, conventions and results of the paper \cite{Z}. 
In particular: 
\begin{itemize} 
	\item 
		$\bbF=\bbR,\bbC$; 
	\item
		$\bbN=\{1,2,3,\dots\}\subset\bbZ_+=\{0,1,2,\dots\}
		\subset\bbZ=\{0,\pm1,\pm2,\dots\}$. 
\end{itemize} 

We, also, use the following notation of homology theory (see, e.g., \cite{SM}): 
\begin{itemize}
	\item 
		$\Hom(\cS;\cS')$ is the set of all mappings from a set $\cS$ to a set $\cS'$; 
	\item 
		$\Hom_\bbF(\cL;\cL')$ is the linear space of all linear mappings from 
		a linear space $\cL$ to a linear space $\cL'$; 			
	\item 
		$\Hom_\cA(\cM;\cM')$ is the $\cA$-module of all 
		$\cA$-linear mappings from an $\cA$-module $\cM$ to an 
		$\cA$-module $\cM'$, where 
		$\cA$ is an associative commutative algebra; 
	\item 
		$\Hom_{\alg}(\cA;\cA')$ is the set of all algebra morphisms  
		from $\cA$ to $\cA'$, where $\cA,\cA'$ are associative commutative 
		algebras; 	
	\item 
		$\Hom_{\Lie}(\fA;\fA')$ is the set of all Lie algebra 
		morphisms from a Lie algebra $\fA$ to a Lie algebra $\fA'$. 
\end{itemize} 

Remind, a set $\fA$ is called a {\it Lie $\cA$-algebra} if it has two structures: 
\begin{itemize} 
	\item 
		the structure of a Lie algebra with a Lie bracket $[\cdot,\cdot]$; 
	\item 
		the structure of an $\cA$-module, where $\cA$ is an associative commutative 
		algebra; 
\end{itemize}
and these structures are related by the matching condition, 
\begin{itemize} 
	\item
		$[X,f\cdot Y]=Xf\cdot Y+f\cdot[X,Y]$ \ for all \ $X,Y\in\fA$ and $f\in\cA$.
\end{itemize} 
We denote by 
$\Hom_{\Lie\cap\cA}(\fA;\fA')=\Hom_{\Lie}(\fA;\fA')\cap\Hom_\cA(\fA;\fA')$ 
the set of all Lie $\cA$-morphisms from a Lie $\cA$-algebra $\fA$ 
to a Lie $\cA$-algebra $\fA'$.

All linear operations are done over the number field $\bbF$. 
The summation over repeated upper and lower indices is as a rule assumed. 
If the corresponding index set is infinite, 
we assume that the summation is correctly defined.
If objects under the study have natural topologies, 
we assume that the corresponding mappings are continuous. 
For example, if $\cS$ and $\cS'$ are topological spaces, 
then $\Hom(\cS;\cS')$ is the set of all continuous mappings from $\cS$ to $\cS'$.  

We use the terminology accepted in the algebra-geometrical approach to 
partial differential equations, because they are the main example 
of the technics developed below. 

\section{Differential algebras.} 
In this section (see \cite{Z},\cite{Z0} for more detail): 
\begin{itemize} 
	\item 
		$\cA$ is an associative commutative algebra; 
	\item 
		$\fM=\fM(\cA)=\End_\cA(\cA)$ is the unital associative algebra 
		of all {\it multiplicators} of the algebra $\cA$; 
	\item 
		$\fD=\fD(\cA)$ is the Lie $\fM$-algebra of all {\it differentiations} 
		of the algebra $\cA$. 
\end{itemize}
\begin{defi} 
A {\it differential algebra} is a pair $(\cA,\cD)$, where 
$\cD=\cD(\cA)$ is a fixed subalgebra ({\it Cartan subalgebra}) of the Lie $\fM$-algebra $\fD=\fD(\cA)$. 
\end{defi}  

\begin{defi} 
A pair $(F,\vfi)$, where the mapping $F\in\Hom_{\alg}(\cA;\cB)$, 
and the mapping $\vfi\in\Hom_{\Lie\cap\fM}(\cD(\cA);\cD(\cB))$, 
is called a {\it morphism} of a differential algebra 
$(\cA,\cD(\cA))$ into a differential algebra $(\cB,\cD(\cB))$, 
if the action $F(Xf)=(\vfi X)(Ff)$ for all $f\in\cA$, $X\in\cD(\cA)$. 
In this case we shall write $(F,\vfi) : (\cA,\cD(\cA))\to(\cB,\cD(\cB))$. 
\end{defi}

Let $(\cA,\cD)$ be a differential algebra. 

\begin{defi} 
A subalgebra $\cB$ (an ideal $\cI$) of the algebra $\cA$ is called {\it differential} 
if $\cD\cB\subset\cB$, i.e., $Xf\in\cB$ for all $X\in\cD$ and $f\in\cB$,  
($\cD\cI\subset\cI$). In this case the pair $(\cB,\cD)$ (the pair $(\cI,\cD)$) 
is a differential algebra.
\end{defi} 
In particular, if $(\cI,\cD)$ is a differential ideal of the differential algebra 
$(\cA,\cD)$, then the quotient differential algebra $(\bcA,\bcD)$ is defined 
by the rule: $\bcA=\cA\big/\cI$, $\bcD=\{\bX=[X]\in\fD(\bcA)\mid X\in\cD\}$, 
$[X][f]=[Xf]$ for all $f\in\cA$, where $[f]=f+\cI,[Xf]=Xf+\cI\in\bcA$. 

\begin{prop} 
For every morphism $(F,\vfi) : (\cA,\cD(\cA))\to(\cB,\cD(\cB))$
the kernel $\ke F$ is a differential ideal of $(\cA,\cD(\cA))$, while  
the image $\im F$ will be a differential subalgebra of $(\cB,\cD(\cB))$ 
if, in addition, the mapping $\vfi : \cD(\cA)\to\cD(\cB)$ is a  surjection. 
\end{prop}

\begin{defi} 
An element $f\in\cA$ is called {\it $\cD$-constant} if 
$\cD f=0$, i.e., $Xf=0$ for all $X\in\cD$.
\end{defi} 
Let $\cA_\cD$ be the set of all $\cD$-constant elements of the algebra $\cA$. 
Clear, $\cA_\cD$ is a subalgebra of $\cA$. 

\begin{prop} 
Let $(F,\vfi) : (\cA,\cD(\cA))\to(\cB,\cD(\cB))$, where $\vfi$ is a surjection. 
Then $F\big|_{\cA_\cD} : \cA_\cD\to\cB_\cD$. 
\end{prop}

\begin{defi} 
A multiplicator  $R\in\fM$ is called {\it $\cD$-constant} if 
$\cD R=[\cD,R]=0$, i.e., $[X,R]=0$ for all $X\in\cD$. 
\end{defi}
Let $\fM_\cD$ be the set of all $\cD$-constant multiplicators of the algebra $\fM$. 
Clear, 
\begin{itemize}
	\item
		$\fM_\cD$ is an unital subalgebra of the algebra $\fM$; 
	\item 
		$\cA_\cD$ is a submodule of the $\fM_\cD$-module $\cA$.  
\end{itemize}

\begin{defi} 
A differentiation $X\in\fD$ is called a {\it Lie-B\"aclund}
differentiation if $[\cD,X]\subset\cD$, i.e., $[Y,X]\in\cD$ for all $Y\in\cD$. 
\end{defi}
Let $\fD_\cD=\fD_\cD(\cA)$ be the set of all Lie-B\"aclund differentiations 
of the differential algebra $(\cA,\cD)$. 
Clear, 
\begin{itemize} 
	\item 
		$\fD_\cD$ is a subalgebra of the Lie $\fM_\cD$-algebra $\fD$. 
\end{itemize}

\begin{prop} 
The ascending filtration 
\begin{equation*}
	\cD=\fD^{(-1)}_\cD\subset\fD_\cD=\fD^{(0)}_\cD\subset\dots 
	\subset\fD^{(q)}_\cD\subset\fD^{(q+1)}_\cD\subset\cdots
\end{equation*}
of Lie $\fM_\cD$-algebras is defined, where 
	$\fD^{(q)}_\cD=\big\{X\in\fD \ \big| \ [\cD,X]\subset\fD^{(q-1)}_\cD\big\}$, 
	$q\in\bbZ_+$. 
Moreover, 
\begin{itemize} 
	\item 
		$\cD$ is an ideal of the Lie $\fM_\cD$-algebra $\fD_\cD$; 
	\item 
		$[\fD^{(p)}_\cD,\fD^{(q)}_\cD]\subset\fD^{(p+q)}_\cD$, 
		$p,q\in\bbZ_+$. 
\end{itemize}
\end{prop}
The general definition of filtration one can find in \cite{RG},\cite{SM}. 
Note, that here and in similar situations below we don't claim that 
	$\cup_{q\in\bbZ_+}\fD^{(q)}_\cD=\fD$  
or 
	$\lim_{q\to\infty}\fD^{(q)}_{\cD}=\fD$
in some sense. 

\begin{defi}\label{D6} 
A differential algebra $(\cA,\cD)$ is called {\it regular} if:
\begin{itemize} 
	\item 
		the Lie $\fM$-algebra $\fD$ is splitted into vertical and horizontal 
		subalgebras, $\fD=\fD_V\oplus_\fM\fD_H$;  
	\item 
		the {\it vertical} subalgebra $\fD_V=\fD_V(\cA)$ has  
		a $\fM$-basis $\p=\{\p_a\in\fD_V\mid a\in\ba\}$, $\ba$ is an 
		index set, $[\p_a,\p_b]=0$, $a,b\in\ba$; 
	\item 
		the {\it horizontal} (Cartan) subalgebra $\fD_H=\fD_H(\cA)=\cD$ 
		has a $\fM$-basis $D=\{D_\mu\in\fD_H\mid \mu\in\bm\}$, 
		$\bm=\{1,\dots,m\}$, $m\in\bbN$, 
		$[D_\mu,D_\nu]=0$, $\mu,\nu\in\bm$; 
	\item 
		the commutators $[D_\mu,\p_a]=\Ga^b_{\mu a}\p_b\in\fD_V$, 
		$\mu\in\bm$, $a,b\in\ba$, the coefficients $\Ga^a_{\mu b}\in\fM$, 
		in particular, 
		$[D_\mu,X]=\na_\mu X=(\na_\mu X)^a\p_a\in\fD_V$ 
		for any $X=X^a\p_a\in\fD_V$, $X^a\in\fM$, where 
		$(\na_\mu X)^a=D_\mu X^a+\Ga^a_{\mu b}X^b$.
\end{itemize}
\end{defi}

Let $(\cA,\fD_H)$ be a regular differential algebra. 

\begin{prop} 
The commutator 
\begin{equation*} 
	[\na_\mu,\na_\nu]^a_b X^b=\big((D_\mu\Ga^a_{\nu b}-D_\nu\Ga^a_{\mu b})
		+(\Ga^a_{\mu c}\Ga^c_{\nu b}-\Ga^a_{\nu c}\Ga^c_{\mu b})\big)X^b, 
		\quad a\in\ba,
\end{equation*} 
or, in the matrix notation, $[\na_\mu,\na_\nu]=F_{\mu\nu}$, where 
\begin{itemize} 
	\item 
		the covariant derivative $\na=(\na_\mu)\in\Hom_\bbF(\fM^\ba;\fM^\ba_\bm)$; 
	\item 
		the connection $\Ga=(\Ga_\mu)$, its components 
		$\Ga_\mu=(\Ga^a_{\mu b})\in\fM^\ba_\ba$; 	  
	\item
		the curvature $F=(F_{\mu\nu})$, its components 
		$F_{\mu\nu}\in\fM^\ba_\ba$; 
	\item 
		$F_{\mu\nu}=D_\mu\Ga_\nu-D_\nu\Ga_\mu+[\Ga_\mu,\Ga_\nu]$, 
		$\mu,\nu\in\bm$.
\end{itemize}
\end{prop}

\begin{exa} 
Here: 
\begin{itemize} 
	\item 
		$\X=\bbR^\bm=\{x=(x^\mu)\mid x^\mu\in\bbR, \ \mu\in\bm\}$ 
		is the space of {\it independent} variables; 
	\item 
		$\bU=\bbR^\ba=\{\bu=(u^a) \mid u^a\in\bbR, \ a\in\ba\}$ 
		is the space of the {\it differential} variables, 
		$\ba$ is an infinite index set; 
	\item 
		$\cA=\cC^\infty_{fin}(\X\bU)$ is the unital associative commutative algebra 
		of $\bbF$-valued smooth functions depending on a finite number of the arguments 
		$x^\mu,u^a$, where $\X\bU=\X\times\bU$.
\end{itemize} 
We split the Lie $\cA$-algebra $\fD=\fD(\cA)$ into vertical and horizontal parts,  
$\fD=\fD_V\oplus_\cA\fD_H$, as follows: 
\begin{itemize} 
	\item 
		$\fD_V$ has the $\cA$-basis $\p=\{\p_{u^a}\mid a\in\ba\}$; 
	\item 
		$\fD_H$ has the $\cA$-basis $D=\{D_\mu\mid \mu\in\bm\}$;
\end{itemize} 
where $D_\mu=\p_{x^\mu}+h^a_\mu\cdot\p_{u^a}$, 
while $\p_{u^a},\p_{x^\mu}$ are partial derivatives. 
The condition $[D_\mu,D_\nu]=0$ to be valid, we assume that the coefficients 
$h^a_\mu\in\cA$ satisfy the equalities: $D_\mu h^a_\nu-D_\nu h^a_\mu=0$ 
for all $a\in\ba$, $\mu,\nu\in\bm$. 
 
In this case, the connection $\Ga=(\Ga_\mu)$ has components 
$\Ga^a_{\mu b}=-\p_{u^b}h^a_\mu$, while the curvature $F=(F_{\mu\nu})=0$. 

From the geometrical point of view the the set $D=\{D_\mu\mid\mu\in\bm\}$ 
defines an involutive (in the Frobenius sense) distribution on the space $\X\bU$. 
Any function $\bphi=(\phi^a(x))\in\cC^\infty(\X;\bU)$ defines a $m$-dimensional 
submanifold $\Phi=\{\bu=\bphi(x)\}$ in $\X\bU$. This submanifold will be 
integral manifold of $D$ if $\big(D_\mu(u^a-\phi^a(x))\big)\big|_{\bu=\bphi(x)}=0$ 
for all $\mu\in\bm$, $a\in\ba$, i.e. the function $\bphi(x)$ 
satisfy the {\it defining system}:
$\p_{x^\mu}\phi^a(x)-h^a_\mu(x,\bphi(x))$, $\mu\in\bm$, $a\in\bA$. 
This system has the integrability condition: 
$\p_{x^\mu}h^a_\nu(x,\bphi(x)=\p_{x^\nu}h^a_\mu(x,\bphi(x)$, 
$\mu,\nu\in\bm$, $a\in\ba$, for any solution $\bphi(x)$ of the system. 
This condition is valid due to assumed equalities 
$D_\mu h^a_\nu-D_\nu h^a_\mu=0$ and the chain rule: 
\begin{equation*} 
	\p_{x^\mu}f(x,\bphi(x))=\big(D_\mu f(x,\bu)\big)\big|_{\bu=\bphi(x)}
\end{equation*}
for all $\mu\in\bm$, $f\in\cA$ and all solutions $\bphi(x)$ of the defining system.
\end{exa}

Now, again let $(\cA,\fD_H)$ be a regular differential algebra.

\begin{prop} 
There is defined the ascending filtration 
\begin{equation*} 
	0\subset\cA^{(0)}_H=\cA_\cD\subset\cA^{(1)}_H\subset\dots
	\subset\cA^{(q)}_H\subset\cA^{(q+1)}_H\subset\cdots
\end{equation*}
of $\fM_\cD$-modules, where  
	$\cA^{(q)}_H=\big\{f\in\cA \ \big| \ D_\mu f\in\cA^{(q-1)}_H, 
	\ \mu\in\bm\big\}$, $q\in\bbN$. 
In particular, 
$\cA^{(p)}_H\cdot\cA^{(q)}_H\subset\cA^{(p+q)}_H$, $p,q\in\bbZ_+$.
\end{prop} 

\begin{prop} 
There is defined the ascending filtration 
\begin{equation*} 
	0\subset\fM^{(0)}_H=\fM_\cD\subset\fM^{(1)}_H\subset\dots
	\subset\fM^{(q)}_H\subset\fM^{(q+1)}_H\subset\cdots 
\end{equation*} 
of $\fM_\cD$-modules, where 
	$\fM^{(q)}_H=\big\{R\in\fM \ \big| \ D_\mu(R)\in\fM^{(q-1)}_H, 
	\ \mu\in\bm\big\}$, $q\in\bbN$. 
In particular, 
$\fM^{(p)}_H\com\fM^{(q)}_H\subset\fM^{(p+q)}_H$, $p,q\in\bbZ_+$.
\end{prop} 

\begin{prop}\label{PE} 
There is defined the ascending filtration 
\begin{equation*}
	0=\cE^{(-1)}\subset\cE=\cE^{(0)}\subset\dots\subset
	\cE^{(q)}\subset\cE^{(q+1)}\subset\cdots
\end{equation*}
of Lie $\fM_\cD$-algebras, where 
	$\cE^{(q)}=\big\{X\in\fD_V \ \big| \ [D_\mu,X]\in\cE^{(q-1)}, \ 	
	\mu\in\bm\big\}$, $q\in\bbN$. 
Moreover, 
\begin{itemize} 
	\item
		$[\cE^{(p)},\cE^{(q)}]\subset\cE^{(p+q)}$, $p,q\in\bbZ_+$; 
	\item 
		$\fD^{(q)}_\cD=\cE^{(q)}\oplus_{\fM_\cD}\cD$, $q\in\bbZ_+$.
\end{itemize}
\end{prop} 

\section{Differential modules.} 
Let $(\cA,\cD)$ be a differential algebra, $\cM$  be an $\cA$-module, 
$\bfM=\fM(\cM)$ be the algebra of all multiplicators of the $\cA$-module $\cM$, 
$\bfD=\fD(\cM)$ be the Lie $\bfM$-algebra of all differentiations of the 
$\cA$-module $\cM$. 

\begin{defi} 
A {\it differential module over a differential algebra} $(\cA,\cD)$ 
(an $(\cA,\cD)$-module) is a triple $(\cM,\ka,\bcD)$, 
where 
\begin{itemize} 
	\item 
		$\cM$ is an $\cA$-module; 
	\item 
		a mapping $\ka\in\Hom_{\Lie\cap\cA}(\fD;\bfD)$, 
		in particular, $\ka[X,Y]=[\ka X,\ka Y]$ for all $X,Y\in\fD$; 
	\item 
		$\bcD$ is the {\it Cartan subalgebra} of the Lie $\bfM$-algebra $\bfD$; 
	\item 
		the matching condition   $\ka\big|_{\cD} : \cD\to\bcD$.
\end{itemize}
\end{defi} 

Let $(\cM,\ka,\bcD)$ be a differential module. 

\begin{defi} 
A submodule $\cN$ of the $\cA$-module $\cM$ is called {\it differential} 
if $\bcD\cN\subset\cN$.
In this case the pair $(\cN,\bcD)$ is a differential module.
\end{defi}

\begin{defi} 
The element $M\in\cM$ is called {\it $\bcD$-constant} if $\bcD M=0$.
\end{defi} 

Let $\cM_{\bcD}$ be the set of all $\bcD$-constant elements of the $\cA$-module $\cM$. 

\begin{defi} 
The multiplicator $\bR\in\bfM$ is called {\it $\bcD$-constant} 
if $\bcD(\bR)=[\bcD,\bR]=0$.
\end{defi} 

Let $\bfM_{\bcD}$ be the set of all $\bcD$-constant elements of the algebra $\bfM$.
Clear, 
\begin{itemize}
	\item
		$\fM_{\bcD}$ is an unital subalgebra of the algebra $\bfM$; 
	\item 
		$\cM_{\bcD}$ is a submodule of the $\fM_{\bcD}$-module $\cM$.  
\end{itemize}

\begin{defi} 
A differentiation $\bX\in\bfD$ is called a {\it Lie-B\"aclund} 
differentiation if $[\bcD,\bX]\subset\bcD$. 
\end{defi}
Let $\bfD_{\bcD}$ be the set of all Lie-B\"aclund differentiations of the differential 
module $(\cM,\bcD)$. 
Clear, 
\begin{itemize} 
	\item 
		$\bfD_{\bcD}$ is a subalgebra of the Lie $\bfM_{\bcD}$-algebra $\bfD$. 
\end{itemize}

\begin{prop} 
There is defined the ascending filtration 
\begin{equation*}
	\bcD=\bfD^{(-1)}_{\bcD}\subset\bfD_{\bcD}=\bfD^{(0)}_{\bcD}\subset\dots 
	\subset\bfD^{(q)}_{\bcD}\subset\bfD^{(q+1)}_{\bcD}\subset\cdots
\end{equation*}
of Lie $\bfM_{\bcD}$-algebras, where 
	$\bfD^{(q)}_{\bcD}=\big\{\bX\in\bfD \ \big| \ [\bcD,\bX]
	\subset\bfD^{(q-1)}_{\bcD}\big\}$, $q\in\bbZ_+$. 
Moreover, 
\begin{itemize} 
	\item 
		$\bcD$ is an ideal of the Lie $\bfM_{\bcD}$-algebra $\bfD_{\bcD}$; 
	\item 
		$[\bfD^{(p)}_{\bcD},\bfD^{(q)}_{\bcD}]\subset\bfD^{(p+q)}_{\bcD}$, 
		$p,q\in\bbZ_+$. 
\end{itemize}
\end{prop} 

\begin{defi} 
Let $(\cA,\cD)$ be a regular differential algebra with a vertical $\fM$-basis 
$\p=\{\p_a\mid a\in\ba\}$ and a horizontal $\fM$-basis $D=\{D_\mu\mid \mu\in\bm\}$.
A differential module $(\cM,\ka,\bcD)$ is called {\it regular} if: 
\begin{itemize} 
	\item 
		the Lie $\bfM$-algebra $\bfD$ is splitted into {\it vertical} and 
		{\it horizontal} subalgebras, 
		$\bfD=\bfD_V\oplus_{\bfM}\bfD_H$; 
	\item 
		the vertical subalgebra $\bfD_V$ has a $\bfM$-basis 
		$\bp=\{\bp_s=(\na_s,\p_s)\in\bfD_V\mid \p_s\in\p, \ s\in\ba_\cM\}$, 
		$[\bp_s,\bp_t]=0$, $s,t\in\ba_\cM$; 
	\item 
		the horizontal subalgebra $\bfD_H=\bcD$ has a $\bfM$-basis 
		$\bD=\{\bD_\si=(\na_\si,D_\si)\in\bfD_H\mid 
		D_\si\in D, \ \si\in\bm_\cM\}$, 
		$[\bD_\si,\bD_\tau]=0$, $\si,\tau\in\bm_\cM$; 
	\item 
		the commutators $[\bD_\si,\bp_s]=\bGa^t_{\si s}\bp_t\in\bfD_V$, 
		$\bGa^t_{\si s}=(\De^t_{\si s},\Ga^t_{\si s})\in\bfM$, 
		in particular 
		$[\bD_\si,\bX]=\bna_\si\bX=(\bna_\si\bX)^s\bp_s\in\bfD_V$ 
		for all $\bX=\bX^s\bp_s\in\bfD_V$, $\bX^s\in\bfM$, where 
		$(\bna_\si\bX)^s=\bD_\si\bX^s+\bGa^s_{\si t}\bX^t$.
\end{itemize}
\end{defi} 

Let $(\cM,\ka,\bfD_H)$ be a regular differential module. 

\begin{prop} 
There is defined the ascending filtration 
\begin{equation*} 
	0\subset\cM^{(0)}_H=\cM_{\bcD}\subset\cM^{(1)}_H\subset\dots
	\subset\cM^{(q)}_H\subset\cA^{(q+1)}_H\subset\cdots
\end{equation*}
of $\bfM_{\bcD}$-modules, where  
	$\cM^{(q)}_H=\big\{M\in\cM \ \big| \ \bD_{\mu'} M\in\cM^{(q-1)}_H, 
	\ \mu'\in\bm'\big\}$, $q\in\bbN$. 
\end{prop} 

\begin{prop} 
There is defined the ascending filtration 
\begin{equation*} 
	0\subset\bfM^{(0)}_H=\bfM_{\bcD}\subset\bfM^{(1)}_H\subset\dots
	\subset\bfM^{(q)}_H\subset\bfM^{(q+1)}_H\subset\cdots 
\end{equation*} 
of $\bfM_{\bcD}$-modules, where 
	$\bfM^{(q)}_H=\big\{\bR\in\bfM \ \big| \ \bD_\mu(\bR)\in\bfM^{(q-1)}_H 
	\ \mu\in\bm\big\}$, $q\in\bbN$. 
In particular, 
$\bfM^{(p)}_H\com\bfM^{(q)}_H\subset\bfM^{(p+q)}_H$, $p,q\in\bbZ_+$.
\end{prop} 

\begin{prop}\label{BE} 
There is defined the ascending filtration 
\begin{equation*}
	0=\bcE^{(-1)}\subset\bcE=\bcE^{(0)}\subset\dots\subset
	\bcE^{(q)}\subset\bcE^{(q+1)}\subset\cdots
\end{equation*}
of Lie $\bfM_{\bcD}$-algebras, where 
	$\bcE^{(q)}=\big\{\bX\in\bfD_V \ \big| \ [\bD_{\mu'},\bX]\subset\bcE^{(q-1)}, \ 	
	\mu'\in\bm'\big\}$, $q\in\bbZ_+$. 
Moreover, 
\begin{itemize} 
	\item
		$[\bcE^{(p)},\bcE^{(q)}]\subset\bcE^{(p+q)}$, $p,q\in\bbZ_+$; 
	\item 
		$\bfD^{(q)}_{\bcD}=\bcE^{(q)}\oplus_{\bfM_{\bcD}}\bcD$, $q\in\bbZ_+$.
\end{itemize}
\end{prop} 

\section{Spectral sequences.} 
\begin{prop}
Let $\cA$ be an unital associative commutative algebra 
(in particular, $\cA$ has a multiplicative unit $e\in\cA$). 
There is the natural isomorphism
\begin{itemize} 
	\item 
		$\cA\simeq\fM(\cA), \quad f\mapsto\ad_f : \cA\to\cA, \quad g\mapsto f\cdot g$. 
\end{itemize}
Moreover, for any $\cA$-module $\cM$ there is the natural isomorphism 
\begin{itemize} 
	\item 
		$\cA\simeq\fM(\cM), \quad f\mapsto\bad_f=(\ad_f,\ad_f), \quad 
		\ad_f : \cM\to\cM,\quad M\mapsto f\cdot M$.
\end{itemize}
\end{prop}
Taking this into account, {\it we shall identify: $\fM(\cA)\equiv\fM(\cM)\equiv\cA$}. 

\begin{ass}\label{A1}
Let $\cU=\cA,\cM$, and $\cV=\cA,\cK$, where 
\begin{itemize} 
	\item 
		$(\cA,\cD)$ is an unital associative commutative algebra; 
	\item 
		$(\cM,\ka=\ka_\cM,\bcD)$ is an $(\cA,\cD)$-module,  
		$\cK$ is an $\cA$-module;
	\item 
		$\ka_{\cK\cM}\in\Hom_{\Lie\cap\cA}(\fD(\cM);\fD(\cK))$, 
		$\ka_\cK\in\Hom_{\Lie\cap\cA}(\fD(\cA);\fD(\cK))$, \\
		$\ka_\cK=\ka_{\cK\cM}\com\ka_\cM$. 
\end{itemize}
\end{ass}

\begin{ass} 
To simplify the notation, below we write $\vk : \fD(\cU)\to\fD(\cV)$, where 
\begin{alignat*}{2} 
	&\vk=\id_{\fD(\cA)}      \text{ \ \  when } \cU=\cV=\cA,  & \quad
	&\vk=\ka_\cK                 \text{ \ \  when } \cU=\cA, \ \cV=\cK,   \\ 
	&\vk=\Pi                         \text{ \ \ when } \cU=\cM, \ \cV=\cA,  &
	& \vk=\vk_{\cK\cM}     \text{ \ \ when } \cU=\cM, \ \cV=\cK,  
\end{alignat*}
the projection $\Pi : \fD(\cM)\to\fD(\cA)$, \ $\bX=(\na_X,X)\mapsto X$ (see \cite{Z}).
\end{ass}

\begin{defi}
The $\cA$-module $\Om(\cU;\cV)=\oplus_{r\in\bbZ}\Om^r(\cU;\cV)$ 
of differential forms over $\cU$ with coefficients in $\cV$ is defined by the rule 
(see \cite{Z}, for example): 
\begin{equation*} 
	\Om^r(\cU,\cV)=
		\begin{cases}  0, &  r<0, \\ \cV, & r=0, \\
					    \Hom_\cA(\w^r_\cA\fD(\cU);\cV), & r>0. 
		\end{cases}
\end{equation*}
The set $\Om(\cU,\cA)$ has the structure of an exterior $\cA$-algebra, 
and the set $\Om(\cU,\cK)$ has the structure of an exterior $\Om(\cU,\cA)$-module. 
Moreover, in general, $\Om(\cU,\cV)=\Om(\cU,\cA)\otimes_\cA\cV$.
\end{defi} 

\begin{prop}\label{P10}
For every $X\in\fD(\cU)$ 
\begin{itemize} 
	\item 
		the {\it interior product} $i_X\in\End_\cA(\Om(\cU,\cV))$ is defined 
		by the {\it contraction rule}   
		$i_X\om(X_1,\dots,X_{r-1})=r\om(X,X_1,\dots,X_{r-1})$  
		for all $r\in\bbZ$, $\om\in\Om^r(\cU,\cV)$, $X_1,\dots,X_{r-1}\in\fD(\cU)$; 
	\item 
		the {\it Lie derivative} $L_X\in\End_\bbF(\Om(\cU,\cV))$ is defined 
		by the rule 
\begin{equation*} 
	L_X\om(X_1,\dots,X_r)=(\vk X)\om(X_1,\dots,X_r)
		-\sum_{1\le s \le r}\om(X_1,\dots[X,X_s]\dots,X_r)
\end{equation*}
	for all $r\in\bbZ$, $\om\in\Om^r(\cU,\cV)$, $X_1,\dots,X_r\in\fD(\cU)$; 
	\item 
		$i_X$ is an exterior differentiation and $L_X$ is a differentiation
		of the $\Om(\cU,\cA)$-module $\Om(\cU,\cV)$, i.e., 
\begin{align*}  
	&i_X\big(\om\w\chi\big)=\big(i_X\om\big)\w\chi+(-1)^r\om\w\big(i_X\chi\big), \\
	&L_X\big(\om\w\chi\big)=\big(L_X\om\big)\w\chi+\om\w\big(L_X\chi\big), 	
\end{align*}
for all $\om\in\Om^r(\cU,\cA)$, $r\in\bbZ_+$, $\chi\in\Om(\cU,\cV)$. 
\end{itemize} 
\end{prop} 
\begin{proof} 
See \cite{Z}, for example.  
\end{proof}

\begin{defi} 
A form $\om\in\Om^r(\cU,\cV)$ is called a {\it $p$-Cartan form}, $0\le p\le r$, 
if $\om(X_1,\dots,X_r)=0$ when at least $r-p+1$ of the differentiations 
$X_1,\dots,X_r$ are Cartan, i.e., belong to the subalgebra $\cD(\cU)\subset\fD(\cU)$.
\end{defi}
Let $\Om^r_p(\cU,\cV)$ be the $\cA$-module of all $p$-Cartan forms in $\Om^r(\cU,\cV)$. 

\begin{prop} 
The descending filtrations 
\begin{equation*} 
	\Om^r_0(\cU,\cV)=\Om^r(\cU,\cV)\supset\Om^r_1(\cU,\cV)\supset\dots
	\supset\Om^r_r(\cU,\cV)\supset 0, \quad r\in\bbZ,
\end{equation*} 
of $\cA$-modules are defined. Moreover, 
\begin{itemize} 
	\item 
		$\Om^r_p(\cU,\cA)\w_\cA\Om^s_q(\cU,\cV)\subset\Om^{r+s}_{p+q}(\cU,\cV)$ 
		for all possible $p,q,r,s\in\bbZ$. 
\end{itemize}
\end{prop}

\begin{prop} 
Let $\om\in\Om^r_p(\cU,\cV)$, $p,r\in\bbZ$, then 
\begin{equation*} 
	i_X\om\in\begin{cases} 
		\Om^{r-1}_{p-1}(\cU,\cV), & X\in\fD(\cU), \\
		\Om^{r-1}_p(\cU,\cV),         & X\in\cD(\cU), \\
		          \end{cases}
		\quad 
	L_X\om\in\begin{cases} 
	       \Om^r_{p-1}(\cU,\cV), & X\in\fD(\cU), \\ 
	       \Om^r_p(\cU,\cV),         & X\in\fD_\cD(\cU). 
	                \end{cases}
\end{equation*}
\end{prop}

Remind (see \cite{Z}, for example), that the mapping  
$d=d_\vk\in\End_\bbF(\Om(\cU,\cV))$ is defined by the {\it  Cartan formula} 
\begin{align*} 
	d\om(X_0,\dots,X_r)
		&=\frac1{r+1}
			\bigg\{\sum_{0\le s\le r}(-1)^s(\vk X_s)\om(x_0,\dots\ck{X_s}\dots,X_r) \\ 
		&+\sum_{0\le s<t\le r}(-1)^{s+t}\om([X_s,X_t],X_0,\dots\ck{X_s}\dots
			\ck{X_t}\dots,X_r)
			\bigg\}
\end{align*}
for all $r\in\bbZ_+$, $\om\in\Om^r(\cU,\cV)$, $X_0,\dots,X_r\in\fD(\cU)$, 
the ``checked'' arguments are understood to be omitted. 

\begin{prop}\label{P12} 
The following statements hold: 
\begin{itemize} 
	\item 
		$d^r=d\big|_{\Om^r(\cU,\cV)} : \Om^r(\cU,\cV)\to\Om^{r+1}(\cU,\cV)$;
	\item 
		the endomorphism $d$ is an exterior differentiation 
		of the exterior $\cA$-algebra $\Om(\cU,\cA)$ 
		and the exterior $\Om(\cU,\cA)$-module $\Om(\cU,\cM)$; 
	\item  
		the composition $d\com d=0$.  
\end{itemize}
\end{prop} 
\begin{proof} See \cite{Z}, Theorem 6. \end{proof}
In particular, the {\it de Rham complex} $\{\Om^r(\cU,\cV),d^r\mid r\in\bbZ\}$ is defined 
with the cohomology spaces $H^r(\cU,\cV)=\ke d^r\big/\im d^{r-1}$, $r\in\bbZ$. 

\begin{prop} 
The Cartan magic formula 
\begin{equation*} 
	L_X=d\com i_X+i_X\com d
\end{equation*} 
holds for any $X\in\fD(\cU)$. 
In particular, $L_X\com d=d\com L_X$ for any $X\in\fD(\cU)$. 
\end{prop}
\begin{proof} 
See \cite{Z}, Theorem 7. 
\end{proof}

The filtrations $\{\Om^r_p(\cU,\cV)\mid p\in\bbZ\}$, $r\in\bbZ$, 
allow one to refine the de Rham complex to a spectral sequence. 

\begin{prop} 
The filtrations  $\{\Om^r_p(\cU,\cV)\}$ are consistent with the differential 
$d\in\End_\bbF(\Om(\cU,\cV))$, namely, 
\begin{equation*} 
	d\big|_{\Om^r_p(\cU,\cV)} : \Om^r_p(\cU,\cV)\to\Om^{r+1}_p(\cU,\cV)  
	\quad\text{for all}\quad  r,p\in\bbZ_+.
\end{equation*}
\end{prop} 

\begin{prop} 
The {\it spectral sequence}\footnote{The general definition see, for example, in \cite{SM},\cite{RG}.} 
$\{\E^{pq}_r,d^{pq}_r\mid p,q,r\in\bbZ\}$ is defined, where 
\begin{itemize} 
	\item 
		$\E^{pq}_r=\Z^{pq}_r\big/\big(\B^{pq}_{r-1}+\Z^{p+1,q-1}_{r-1}\big)$; 
	\item 
		$\Z^{pq}_r=\big\{\om\in\Om^{p+q}_p(\cU,\cV) \ \big| \ 
		d\om\in\Om^{p+q+1}_{p+r}(\cU,\cV)\big\}$; 
	\item 
		$\B^{pq}_r=\big\{\om=d\chi\in\Om^{p+q}_p(\cU,\cV) \ \big| \ 
		\chi\in\Om^{p+q-1}_{p-r}(\cU,\cV)\big\}$; 
	\item 
		$d^{pq}_r : \E^{pq}_r\to\E^{p+r,q-r+1}_r$, \quad 
		$d^{pq}_r[\om]^{pq}_r=[d\om]^{p+r,q-r+1}_r$, \\
		where $[\om]^{pq}_r$ is the equivalence class of a form $\om\in\Z^{pq}_r$ 
		in the quotient space $\E^{pq}_r$.
\end{itemize}
\end{prop}

In particular, 
\begin{itemize} 
	\item 
		$\E^{pq}_r=\Om^{p+q}_p(\cU,\cV)\big/\Om^{p+q}_{p+1}(\cU,\cV)$ 
		for all $r\leq 0$, $p,q\in\bbZ$; 
	\item 
		$\E^{pq}_r=0$ for all $p<0$, $r,q\in\bbZ$, and for all $q<0$, $r,p\in\bbZ$; 
	\item 
		$\E^{pq}_{r+1}=\ke d^{pq}_r\big/\im d^{p-r,q+r-1}_r$ for all $p,q,r\in\bbZ$.
\end{itemize} 

\begin{prop} 
The limit terms of the spectral sequence $\{\E^{pq}_r,d^{pq}_r\mid p,q,r\in\bbZ\}$ 
are defined as follows: 
\begin{itemize} 
	\item 
		$\E^{pq}_\infty
		=\Z^{pq}_\infty\big/\big(\B^{pq}_\infty+\Z^{p+1,q-1}_\infty\big)$; 
	\item 
		$\Z^{pq}_\infty=\big\{\om\in\Om^{p+q}_p(\cU,\cV) \ \big| \ d\om=0\big\}$; 
	\item 
		$\B^{pq}_\infty=\big\{\om=d\chi\in\Om^{p+q}_p(\cU,\cV) \ \big| \ 
		\chi\in\Om^{p+q-1}(\cU,\cV)\big\}$.
\end{itemize}
\end{prop}

\begin{defi} 
The filtration $\{\Om^r_p(\cU,\cV)\}$ is called {\it Cartan} if there exists a number 
$m\in\bbN$, s.t., 
\begin{equation*} 
	\Om^r_0(\cU,\cV)=\dots=\Om^r_{r-m}(\cU,\cV)
	\supset\Om^r_{r-m+1}(\cU,\cV)\supset\dots\supset\Om^r_r(\cU,\cV)\supset 0.
\end{equation*}
This number $m$ is called the {\it Cartan dimension} of the Lie $\cA$-algebra $\cD(\cU)$. 
\end{defi} 

\begin{prop} 
Let the filtration  $\{\Om^r_p(\cU,\cV)\}$ be Cartan.
Then the limit equalities 
\begin{equation*} 
	\lim_{r\to\infty}\E^{pq}_r=\E^{pq}_\infty \quad\text{hold for all}\quad  p,q\in\bbZ.
\end{equation*}
\end{prop}
\begin{proof} 
The existence of the limit follows from the general properties of spectral sequences, 
see \cite{RG} for full detail.
\end{proof}

In applications, as a rule, only separate terms of the spectral sequence are used. 
Thus, in the algebra-geometrical approach to partial differential equations 
the Cartan spectral sequence $\{\E^{pg}_r,d^{pq}_r\mid p,r\in\bbZ_+,0\le q\le m\}$ 
($m$ is the Cartan dimension) arises (see, \cite{AV} and, for example, \cite{Z6}) 
and the important role play the following terms: 
\begin{itemize}  
	\item 
		$\E^{0m}_1$ is the space of {\it functionals}, 
		elements of the equivalence classes are called {\it Lagrangians};
	\item 
		$\E^{0,m-1}_1$ is the space of {\it conservation laws}, 
		elements of the equivalence classes are called {\it conserved currents}; 
	\item 
		$\E^{0,q}_1$, $0\le q\le m-2$, are the spaces of 
		{\it conservation laws of lower order};
	\item 
		$\E^{pm}_1$, $p\in\bbN$, are the spaces of the {\it functional $p$-forms}; 
	\item 
		$d^{pm}_1 : \E^{pm}_1\to\E^{p+1,m}_1$, $p\in\bbZ_+$, 
		are {\it functional (variation) differentials}, 
		the differential $\de=d^{0m}_1$ is called the {\it Euler-Lagrange operator}.
\end{itemize}

\begin{defi} 
The quotient Lie algebra of {\it symmetries} of  a differential algebra 
(a differential module)  $\cU$ is defined as $\Sym_\cD\cU=\fD_\cD(\cU)\big/\cD(\cU)$.
\end{defi}

\begin{prop} 
For any $[X]=X+\cD(U)\in\Sym_\cD\cU$, $X\in\fD_\cD(\cU)$, 
the following quotient morphisms are defined: 
\begin{itemize} 
	\item 
		$\big(i_{[X]}\big)_r : \E^{pq}_r\to\E^{p-1,q}_r$, \quad 
		$[\om]^{pq}_r\mapsto[i_X\om]^{p-1,q}_r$; 
	\item 
		$\big(L_{[X]}\big)_1 :  \E^{pq}_1\to\E^{pq}_1$, \quad 
		$[\om]^{pq}_1\mapsto[L_X\om]^{pq}_1$.
\end{itemize}
\end{prop}

\section{Variation bicomplexes.}\label{S5}
Here we keep all notations and assumptions of the previous section. 
\begin{ass} 
In addition to the Assumption \ref{A1}, 
assume that the differential algebra $(\cA,\cD)$ and the differential module 
$(\cM,\ka=\ka_\cM,\bcD)$ are regular. 
In this case the Cartan dimensions are 
$m_\cA=\dim_\cA\cD$, $m_\cM=\dim_\cA\bcD$. 
\end{ass}

The splitting of the algebras $\fD(\cA)$ and $\fD(\cK)$  
into vertical and horizontal parts allows to refine the spectral sequence 
into the variation bicomplex (the detailed information about bicomplexes see, 
for example, in \cite{SM} or \cite{Z3}. 

\begin{defi} 
The $\cA$-modules $\Om^{pq}(\cU,\cV)$, $p,q\in\bbZ$,  are defined by the rule 
\begin{equation*} 
	\Om^{pq}(\cU,\cV)=
	\begin{cases}0,& p<0 \text{ and/or } q<0, \\ 
	 \cV, & p=q=0, \\ 
	\Hom_\cA((\w^p_\cA\fD_V(\cU))\w_\cA(\w^q_\cA\fD_H(\cU));\cV), &  p+q>0.  
	\end{cases} 
\end{equation*}
\end{defi} 
In particular, 
$\Om^p_V(\cU;\cV)=\Om^{p0}(\cU;\cV)$, $\Om^q_H(\cU;\cV)=\Om^{0q}(\cU;\cV)$ 
are the $\cA$-modules of  the {\it vertical} and the {\it horizontal} forms, correspondingly. 
Moreover, in general, 
$\Om^{pq}(\cU;\cV)=\Om^p_V(\cU;\cA)\w_\cA\Om^q_H(\cU;\cA)\otimes_\cA\cV$. 

Note, $\Om^q_H(\cU,\cV)=0$ for all $q>m_\cU$.

The splitting $\fD(\cU)=\fD_V(\cU)\oplus_\cA\fD_H(\cU)$ defines 
the projections 
\begin{equation*}
	\pi_{V,H} : \fD(\cU)\to \fD_{V,H}(\cU), \quad
	X=X_V+X_H\mapsto\pi_{V,H}X=X_{V,H}. 
\end{equation*}
Hence, there are defined the dual injections 
\begin{equation*}
	\pi^*_{V,H} : \Om^r_{V,H}(\cU,\cV)\to\Om^r(\cU,\cV),
	\quad\om\mapsto\pi^*_{V,H}\om=\om\com\pi_{V,H}, 
\end{equation*}
where 
$\pi^*_{V,H}\om(X_1,\dots,X_r)=\om(\pi_{V,H}X_1,\dots,\pi_{V,H}X_s)$, 
$X_1,\dots,X_r\in\fD(\cU)$. 

\begin{prop} 
The identifications $\pi^*_{V,H}\om=\om$ define the representations 
\begin{equation*} 
	\Om^r(\cU,\cV)=\oplus_{s\in\bbZ}\Om^{s,r-s}(\cU,\cV), 
	\quad\text{and}\quad 
	\Om^r_p(\cU,\cV)=\oplus_{s\ge p}\Om^{s,r-s}(\cU,\cV).
\end{equation*}
In particular, the $\cA$-module 
$\Om(\cU,\cV)=\oplus_{p,q\in\bbZ}\Om^{pq}(\cU,\cV)$ 
is bigraded. 
\end{prop} 

By the previous assumptions, the unital differential algebra $(\cA,\cD)$ is regular, 
i.e., $\fD(\cA)=\fD_V(\cA)\oplus_\cA\fD_H(\cA)$, 
the vertical subalgebra $\fD_V(\cA)$ has an $\cA$-basis $\p=\{\p_a\mid a\in\ba\}$, 
$[\p_a,\p_b]=0$,
while the horizontal subalgebra $\fD_H(\cA)=\cD$ has an $\cA$-basis 
$D=\{D_\mu\mid\mu\in\bm\}$, $[D_\mu,D_\nu]=0$. 

In this case, the dual $\cA$-module 
$\fD^*_V(\cA)=\Hom_\cA(\fD_V(\cA);\cA)=\Om^1_V(\cA,\cA)$ 
has the dual $\cA$-basis $\rho=\{\rho^a\in\fD^*_V(\cA)\mid a\in\ba\}$, 
$\rho^a(\p_b)=\de^a_b$. 
In particular, $\om(X)=\om_a\rho^a(X^b\p_b)=\om_aX^a$ for all 
$\om\in\fD^*_V(\cA)$, $X\in\fD_V(\cA)$, while $\om(X)=0$ for any 
$\om\in\fD^*_V(\cA)$, $X\in\fD_H(\cA)$. 

The dual $\cA$-module 
$\fD^*_H(\cA)=\Hom_\cA(\fD_H(\cA);\cA)=\Om^1_H(\cA,\cA)$ 
has the dual $\cA$-basis $\vte=\{\vte^\mu\mid \mu\in\bm\}$, 
$\vte^\mu(D_\nu)=\de^\mu_\nu$. 
In particular, $\om(X)=\om_\mu\vte^\mu(X^\nu D_\nu)=\om_\mu X^\mu$ 
for all $\om\in\fD^*_H(\cA)$, $X\in\fD_H(\cA)$, 
while $\om(X)=0$ for any $\om\in\fD_H(\cA)$, $X\in\fD_V(\cA)$.

The same is true for the $(\cA,\cD)$-module $(\cM,\ka=\ka_\cM,\cD)$. 
Namely\footnote{We simplify the notation, for convenience.}, 
$\fD(\cM)=\fD_V(\cM)\oplus_\cA\fD_H(\cM)$, 
the vertical subalgebra $\fD_V(\cM)$ has an $\cA$-basis 
$\bp=\{\bp_s\mid s\in\ba_\cM\}$, $[\bp_s,\bp_t]=0$. 
Further, the horizontal subalgebra $\fD_H(\cM)$ has an $\cA$-basis 
$D=\{D_\si\mid \si\in\bm_\cM\}$, $[\bD_\si,\bD_\tau]=0$.  

Thus, the dual $\cA$-module 
$\fD^*_V(\cM)=\Hom_\cA(\fD_V(\cM);\cA)=\Om^1_V(\cM,\cA)$ 
has the dual $\cA$-basis 
$\brho=\{\brho^s\mid s\in\ba_\cM\}$, $\brho^s(\bp_t)=\de^s_t$.  
In particular, $\om(X)=\om_s\brho^s(X^t\bp_t)=\om_s X^s$ for all 
$\om\in\fD^*_V(\cM)$, $X\in\fD_V(\cM)$, while $\om(X)=0$ for any 
$\om\in\fD^*_V(\cM)$, $X\in\fD_H(\cM)$. 

The dual $\cA$-module  
$\fD^*_H(\cM)=\Hom_\cA(\fD_H(\cM);\cA)=\Om^1_H(\cM,\cA)$ 
has the dual $\cA$-basis $\bvte=\{\bvte^\si\mid \si\in\bm_\cM\}$, 
$\bvte^\si(\bD_\tau)=\de^\si_\tau$. 
In particular, $\om(X)=\om_\si\bvte^\si(X^\tau \bD_\tau)=\om_\si X^\si$ 
for all $\om\in\fD^*_H(\cM)$, $X\in\fD_H(\cM)$, 
while $\om(X)=0$ for any $\om\in\fD_H(\cM)$, $X\in\fD_V(\cM)$. 

\begin{ass} 
Below we further simplify the notation and write: 
$\p,D$ for the vertical and the horizontal bases in $\fD(\cU)$ 
and $\rho,\vte$ for the dual bases. 
\end{ass}

\begin{prop}\label{P17} 
The representations  
\begin{equation*} 
	\Om^{pq}(\cU,\cV)
	=\bigg\{\om=\frac1{p!q!}\om_{a_1\dots a_p\mu_1\dots\mu_q}
	\cdot\rho^{a_1}\w\dots\w\rho^{a_p}\w\vte^{\mu_1}\w\dots\w\vte^{\mu_q}\bigg\}
\end{equation*}
hold for all $p\in\bbZ_+$ and $0\le q\le m$, where the coefficients 
$\om_{a_1\dots a_p\mu_1\dots\mu_q}\in\cV$ are skew-symmetric 
in indices $a_1,\dots,\mu_q$.
\end{prop} 

The exterior differentiation 
$d : \Om^{pq}(\cU,\cV)\to\Om^{p+1,q}(\cU,\cV)\oplus_\bbF\Om^{p,q+1}(\cU,\cV)$ 
also splits into {\it vertical} and {\it horizontal} parts, $d=d_V+d_H$, where 
\begin{equation*} 
	d_V : \Om^{pq}(\cU,\cV)\to\Om^{p+1.q}(\cU,\cV), \quad 
	d_H : \Om^{pq}(\cU,\cV)\to\Om^{p,q+1}(\cU,\cV).
\end{equation*} 

Indeed, the Cartan formula gives the following results: 
\begin{itemize} 
	\item 
		$dv=\vk\p_a v\cdot\rho^a+\vk D_\mu v\cdot\vte^\mu$ for any $v\in\cV$, 
		hence \\
		$d_V v=\vk\p_a v\cdot\rho^a$, \quad $d_H v=\vk D_\mu v\cdot\vte^\mu$; 
	\item 
		$d\rho^a=\Ga^a_{\mu b}\cdot\rho^b\w\vte^\mu$ for any $\rho^a\in\rho$, 
		hence \\
		$d_V\rho^a=0$, \quad $d_H\rho^a=\Ga^a_{\mu b}\cdot\rho^b\w\vte^\mu$; 
	\item 
		$d\vte^\mu=0$ for any $\vte^\mu\in\vte$, hence $d_{V,H}\vte^\mu=0$.
\end{itemize} 
Due to Propositions \ref{P12} and \ref{P17} these formulas allow to calculate 
differentials $d_V\om$ and $d_H\om$ for any form $\om\in\Om^{pq}(\cU,\cV)$. 

\begin{prop} 
The equality $d\com d=0$ implies the equalities:
\begin{equation*} 
	d_V\com d_V=d_V\com d_H+d_H\com d_V=d_H\com d_H=0.
\end{equation*}
\end{prop}

\begin{prop} 
There is defined the variation bicomplexe 
\begin{equation*}
	\{\Om^{pq}(\cU,\cV);d^{pq}_V,d^{pq}_H\mid p,q\in\bbZ\}, 
	\quad\text{where}\quad d^{pq}_{V,H}=d_{V,H}|_{\Om^{pq}(\cU,\cV)}. 
\end{equation*} 
\end{prop}
 
The vertical and horizontal cohomologies of these bicomplexes are defined as follows: 
\begin{equation*}
	H^{pq}_V(\cU,\cV)=\ke d^{pq}_V\big/\im d^{p-1,q}_V, \quad  
	H^{pq}_H(\cU,\cV)=\ke d^{pq}_H\big/\im d^{p,q-1}_H, 
\end{equation*}
for all $p,q\in\bbZ$. 

To shorten the notation, below in this section  we omit the arguments $\cU$ and $\cV$ 
and write 
$\Om^{pq}$ instead of $\Om^{pq}(\cU,\cV)$, 
$H^{pq}_{V,H}$ instead of $H^{pq}_{V,H}(\cU,\cV)$, $\fD$ instead of $\fD(\cU)$, 
and so on. 

For all $p,q\in\bbZ$ there are defined the second differentials: 
\begin{itemize} 
	\item 
		$d^{pq}_{HV}\in\Hom_\bbF(H^{pq}_V,H^{p,q+1}_V)$, \quad
		$\om^{pq}_V\mapsto d^{pq}_{HV}\om^{pq}_V=[d_H\om^{pq}]_V$, 
	\item 
		$d^{pq}_{VH}\in\Hom_\bbF(H^{pq}_H,H^{p+1,q}_H)$, \quad
		$\om^{pq}_H\mapsto d^{pq}_{VH}\om^{pq}_H=[d_V\om^{pq}]_H$, 
\end{itemize}
where 
\begin{itemize}
	\item
		$\om^{pq}_V=[\om^{pq}]_V=\om^{pq}+d_V\Om^{p-1,q}\in H^{pq}_V$, 
		\ $\om^{pq}\in\Om^{pq}$, \ $d_V\om^{pq}=0$, 
	\item
		$\om^{pq}_H=[\om^{pq}]_H=\om^{pq}+d_H\Om^{p,q-1}\in H^{pq}_H$, 
		\ $\om^{pq}\in\Om^{pq}$, \ $d_H\om^{pq}=0$.
\end{itemize}
Thus, there are defined the complexes 
\begin{equation*}
	\{H^{pq}_V;d^{pq}_{HV}\mid q\in\bbZ\}, \ p\in\bbZ, 
	\quad
	\{H^{pq}_H;d^{pq}_{VH}\mid p\in\bbZ\}, \ q\in\bbZ, 
\end{equation*}
with the second cohomology spaces 
\begin{align*} 
	H^{pq}_{HV}&=\ke d^{pq}_{HV}\big/\im d^{p,q-1}_{HV} \\ 
	                        &=\{\om^{pq}_{HV}
	=\om^{pq}+d_V\Om^{p-1,q}\mid \om^{pq}\in\Om^{pq},
	d_V\om^{pq}=0, d_H\om^{pq}\in d_V\Om^{p-1,q+1}\}, \\
	H^{pq}_{VH}&=\ke d^{pq}_{VH}\big/\im d^{p-1,q}_{VH} \\ 
	                        &=\{\om^{pq}_{VH}
	=\om^{pq}+d_H\Om^{p,q-1}\mid 
	\om^{pq}\in\Om^{pq}, d_H\om^{pq}=0,d_V\om^{pq}\in d_H\Om^{p+1,q-1}\}, 
\end{align*} 
All further differentials and cohomologies are trivial. 

\begin{theorem} 
In the regular case elements of the spectral sequence $\{E^{pq}_r,d^{pq}_r\}$ 
are as follows: 
\begin{itemize} 
	\item 
		$E^{pq}_0=\Om^{p+q}_p\big/\Om^{p+q}_{p+1}=\Om^{pq}$, 
		$d^{pq}_0=d^{pq}_H : \Om^{pq}\to\Om^{p,q+1}$; 
	\item 
		$E^{pq}_1=\ke d^{pq}_H\big/\im d^{p,q-1}_H=H^{pq}_H$, 
		$d^{pq}_1=d^{pq}_{VH} : H^{pq}_H\to H^{p+1,q}_H$; 
	\item 
		$E^{pq}_2=\ke d^{pq}_{VH}\big/\im d^{p-1,q}_{VH}=H^{pq}_{VH}$, 
		$d^{pq}_2=0$; 
	\item 
		$E^{pq}_r=E^{pq}_2=H^{pq}_{VH}$, $d^{pq}_r=0$, $r\ge2$;  
	\item 
		$\lim_{r\to\infty}E^{pq}_r=E^{pq}_2=E^{pq}_\infty$. 
\end{itemize}
\end{theorem} 
\begin{proof}
The proof is based on the general properties of spectral sequences and the 
above calculations for the variation bicomplex. 
\end{proof}

Let $X=X_V+X_H=X^a\p_a+X^\mu D_\mu\in\fD=\fD_V+\fD_H$. Then, 
\begin{itemize} 
	\item 
		$i_X F=0$ for any $F\in\cV$; 
	\item 
		$i_X\rho^a=X^a=i_{X_V}\rho^a$ for any $a\in\ba$; 
	\item 
		$i_X\vte^\mu=X^\mu=i_{X_H}\vte^\mu$ for any $\mu\in\bm$; 
	\item 
		$L_X F=(\vk X)F$ for any $F\in\cV$; 
	\item 
		$L_X\rho^a=(\p_b X^a-\Ga^a_{\mu b}X^\mu)\cdot\rho^b
		+(\na_\mu X_V)^a\cdot\vte^\mu$ for any $a\in\ba$; 
	\item 
		$L_X\vte^\mu=\p_a X^\mu\cdot\rho^a+D_\nu X^\mu\cdot\vte^\nu$ 
		for any $\mu\in\bm$. 
\end{itemize}
Remind, $(\na_\mu X)^a=D_\mu X^a+\Ga^a_{\mu b}X^b=(\na_\mu X_V)^a$, 
see Definition \ref{D6}. 
Due to Propositions \ref{P10} and \ref{P17} these formulas allow to calculate 
$i_X\om$ and $L_X\om$ for any form $\om\in\Om^{pq}$. 

\begin{prop} 
Let $\om\in\Om^{pq}$, $p,q\in\bbZ$, $X\in\fD$, 
then 
\begin{equation*} 
	i_X\om\in\begin{cases}
		\Om^{p-1,q}, & X\in\fD_V, \\
		\Om^{p,q-1}, & X\in\fD_H, \\ 
			  \end{cases} 
		\quad 
	L_X\om\in\begin{cases} 
		\Om^{p+1,q-1}\oplus\Om^{pq}\oplus\Om^{p-1,q+1}, & X\in\fD, \\ 
		\Om^{pq}, & X\in\cE. 
			    \end{cases}
\end{equation*} 
\end{prop}
Remind, $\cE\subset\fD$, see Propositions \ref{PE} and \ref{BE}. 

\begin{prop} 
For any $X\in\cE$, $p,q\in\bbZ$, the endomorphisms 
\begin{equation*}
	[L_X]^{pq}\in\End_\bbF(H^{pq}_{V,H}), \quad 
	[\om^{pq}]_{V,H}\mapsto[L_X]^{pq}[\om^{pq}]_{V,H}=[L_X\om^{pq}]_{V,H},  
\end{equation*}
are defined, where 
$[\om^{pq}]_V=\om^{pq}+d^{p-1,q}_V\Om^{p-1,q}$, 
$[\om^{pq}]_H=\om^{pq}+d^{p,q-1}_H\Om^{p,q-1}$.
\end{prop}

\section{Differential algebras in partial differential equations} 

\subsection{Notation} 
Here: 
\begin{itemize} 
	\item 
		$\X=\bbR^\bm=\{x=(x^\mu)\mid x^\mu\in\bbR, \ \mu\in\bm\}$ 
		is the linear space of {\it independent} variables; 
	\item 
		$\U=\bbR^\A=\{u=(u^\al)\mid u^\al\in\bbR, \ \al\in\A\}$ 
		is the linear space of {\it dependent} variables, 
		$\A$ is a finite index set; 
	\item 
		$\bU=\bbR^\A_\I=\{\bu=(u^\al_i)\mid u^a_i\in\bbR, \ \al\in\A, \ i\in\I\}$ 
		is the linear space of {\it differential} variables, 
		$\I=\bbZ^\bm_+$ (note, $\dim\bU=\infty$); 
	\item 
		$\cA=\cC^\infty_{fin}(\X\bU)$ is the unital associative commutative algebra 
		of $\bbF$-valued smooth functions depending on a finite number 
		of the arguments $x^\mu,u^\al_i$, 
		$\X\bU=\X\times\bU$.
\end{itemize}
In this case, $\fM(\cA)=\cA$, because the algebra $\cA$ is unital. 
The Lie $\cA$-algebra $\fD=\fD(\cA)$ has the standard $\cA$-basis 
$\{\p_{u^\al_i},\p_{x^\mu}\mid\al\in\A, \ i\in\I, \ \mu\in\bm\}$, 
where $\p_{u^\al_i},\p_{x^\mu}$ are partial derivatives. 

In the algebraic approach to partial differential equations the Lie $\cA$-algebra 
$\fD$ splits as  $\fD=\fD_V\oplus_\cA\fD_H$, where 
\begin{itemize}
	\item 
		the vertical subalgebra $\fD_V$ has the $\cA$-basis 
		$\p=\{\p_{u^\al_i}\mid \al\in\A, \ i\in\I\}$, $[\p_{u^\al_i},\p_{u^\be_j}]=0$; 
	\item 
		the horizontal subalgebra $\fD_H$ has the $\cA$-basis 
		$D=\{D_\mu\mid \mu\in\bm\}$, \ 
		$D_\mu=\p_{x^\mu}+u^\al_{i+(\mu)}\p_{u^\al_i}$, \ 
		$i+(\mu)=(i^1,\dots,i^\mu+1,\dots,i^m)$, \ $[D_\mu,D_\nu]$=0.
\end{itemize}
The horizontal basic differentiations $D_\mu$ are called {\it total derivatives}, 
they are characterized by the {\it chain rule}: 
\begin{equation*} 
	\p_{x^\mu}\big(f(x,\bu)\big|_{\bu=\bphi(x)}\big)
	=\big(D_\mu f(x,\bu)\big)\big|_{\bu=\phi(x)}
\end{equation*}
for all $\mu\in\bm$, $f\in\cA$ and $\phi\in\cC^\infty(\X;\U)$, 
where $\phi(x)=(\phi^\al(x))$, $\bphi(x)=(\phi^\al_i(x))$, 
$\phi^\al_i(x)=\p_{x^i}\phi^\al(x)$, 
$\p_{x^i}=(\p_{x^1})^{i^1}\dots(\p_{x^m})^{i^m}$, 
$i=(i^1,\dots,i^m)\in\I$. 
\begin{itemize} 
	\item 
		The commutators 
		$[D_\mu,\p_{u^\al_i}]=-\p_{u^\al_{i-(\mu)}}$, 
		hence the connection $\Ga=(\Ga^{i\be}_{\mu\al j})$, 
		$\Ga^{i\be}_{\mu\al j}=-\de^\be_\al\de^i_{j+(\mu)}$, 
		$\mu\in\bm$, $\al,\be\in\A$, $i,j\in\I$. 
\end{itemize} 
Thus, the regular unital differential algebra $(\cA,\fD_H)$ is defined.  

Here, the commutator $[D_\mu,X]=(\na_\mu\zeta)^\al_i\p_{u^\al_i}$ 
for any $X=\zeta^\al_i\p_{u^\al_i}\in\fD_V$, where 
\begin{equation*}
	\na_\mu\in\End_\bbF(\cA^\A_\I), \quad  
	\zeta=(\zeta^\al_i)\mapsto\na_\mu\zeta=((\na_\mu\zeta)^\al_i), \quad 
	(\na_\mu\zeta)^\al_i=D_\mu\zeta^\al_i-\zeta^\al_{i+(\mu)}. 
\end{equation*}

\begin{prop}\label{PC} 
The commutator $[\na_\mu,\na_\nu]=0$ for all $\mu,\nu\in\bm$, i.e., 
the curvature $F=(F_{\mu\nu})=0$.
\end{prop}

\begin{defi} 
The number $n\in\bbZ_+$ is called the {\it order} of a differential function 
$f\in\cA$, if $\p_{u^\al_i}f\ne0$ for some $\al\in\A$ and $i\in\I$, 
$|i|=i^1+\dots i^m=n$, while $\p_{u^\be_j}f=0$ 
for all $\be\in\A$ and $j\in\I$, $|j|=j^1+\dots j^m>n$.
\end{defi} 

\begin{lemma}\label{L1} 
The subalgebra $\cA_\cD=\cA^{(0)}_H=\bbF$. 
\end{lemma}
\begin{proof}
The proof is based on the property that by definition every differential function 
$f\in\cA$ has a finite order $n=n(f)$.  
\end{proof} 

\begin{prop} 
In the filtration $\{\cA^{(q)}_H\mid q\in\bbZ_+\}$ the linear spaces 
\begin{equation*}
	\cA^{(q)}_H=\bbF_q[x]=\bigg\{f(x)=\sum_{|i|\le q}f_ix^i \ \bigg| \ f_i\in\bbF\bigg\} 
\end{equation*}
are the spaces of polynomials of the order $q$ in $x\in\X$, where 
$x^i=(x^1)^{i^1}\!\dots\!(x^m)^{i^m}$, $i\in\I$, $|i|=i^1+\dots+i^m$.
\end{prop}

The limit $\lim_{q\to\infty}\cA^{(q)}_H$ depends on the topology. 
Thus, if we choose the natural topology of the linear space of polynomials $\bbF[x]$
then we get $\lim_{q\to\infty}\cA^{(q)}_H=\bbF[x]$,
while if we choose the natural topology of the linear linear space of smooth functions 
$\cC^\infty(\X;\bbF)$ then we get $\lim_{q\to\infty}\cA^{(q)}_H=\cC^\infty(\X;\bbF)$. 

\begin{lemma}\label{L2} 
The equalities $(\na_r\zeta)^\al_i=\sum_{k+j=r}(-1)^k\binom{r}{k}D_j\zeta^\al_{i+k}$ 
hold, where $\al\in\I$,  $i,r,k,j\in\I$, $\na_r=(\na_1)^{r^1}\dots(\na_m)^{r^m}$, 
$D_j=(D_1)^{j^1}\dots(D_m)^{j^m}$.
\end{lemma}
\begin{proof} 
The proof is based on Proposition \ref{PC}, the induction on $r$ and the well 
known equality $\binom{r}{k}+\binom{r}{k-(\mu)}=\binom{r+(\mu)}{k}$. 
\end{proof} 
\begin{rem} 
We use the standard multiindex notation, 
in particular, $(-1)^r=(-1)^{|r|}=(-1)^{r^1+\dots+r^m}$, 
$\binom{r}{k}=\binom{r^1}{k^1}\dots\binom{r^m}{k^m}$.
\end{rem} 

\begin{defi} 
For every $k\in\I$ we define the linear subspace 
\begin{equation*} 
	\Phi^k=\big\{\eps^k_\phi=(\eps^{k\al}_{\phi i})\in\cA^\A_\I \ \big| \ 
	\eps^{k\al}_{\phi i}=\tbinom{i}{k}D_{i-k}\phi^\al, \ \phi=(\phi^\al)\in\cA^\A\big\}
	\subset\cA^\A_\I.
\end{equation*}
We also set $\Phi^k=0$ if $k\notin\I$.
\end{defi}

\begin{lemma}\label{L3} 
For any $r,k\in\I$ the mapping 
\begin{equation*} 
	\na_r\in\Hom_\bbF(\Phi^k;\Phi^{k-r}), \quad 
	\eps^k_\phi\mapsto\na_r\eps^k_\phi=(-1)^r\eps^{k-r}_\phi.
\end{equation*}
In particular, $\na_r\eps^k_\phi=0$ for any $k-r\notin\I$ and $\phi\in\cA^\A$.
\end{lemma} 
\begin{proof} 
Indeed, 
\begin{align*} 
	(\na_\mu\eps^k_\phi)^\al_i
		&=D_\mu\eps^{k\al}_{\phi i}-\eps^{k\al}_{\phi,i+(\mu)}
	=\tbinom{i}{k}D_{i+(\mu)-k}\phi^\al-\tbinom{i+(\mu)}{k}D_{i+(\mu)-k}\phi^\al \\
	       &=\big(\tbinom{i}{k}-\tbinom{i+(\mu)}{k}\big)D_{i+(\mu-k)}\phi^\al 
	=-\tbinom{i}{k-(\mu)}D_{i-k+(\mu)\phi}=-\eps^{k-(\mu),\al}_{\phi i}.
\end{align*}
To complete the proof one should use induction on $r$.
\end{proof} 

\begin{theorem} 
For any $\zeta=(\zeta^\al_i)\in\cA^\A_\I$ there exists the unique representation 
\begin{equation*} 
	\zeta=\sum_{k\in\I}\eps^k_{\phi_k}, \quad \phi_k=(\phi^\al_k)\in\cA^\A, 
	\quad \phi^\al_k=\sum_{i+j=k}(-1)^j\tbinom{k}{i}D_j\zeta^\al_i.
\end{equation*} 
In other words, the linear space $\cA^\A_\I$ is $\I$-graded by the linear 
spaces $\Phi^k$, i.e., $\cA^\A_\I=\oplus_{k\in\I}\Phi^k$.
\end{theorem} 
\begin{proof} 
Indeed, for a given $\zeta\in\cA^\A_\I$ we need to find functions 
$\phi_k\in\cA^\A$, $k\in\I$, satisfying the equality $\zeta=\sum_{k\in\I}\eps^k_{\phi_k}$. 
Applying the operator $\na_r$, $r\in\I$, to both sides of this equality we get 
(see Lemmas \ref{L2} and \ref{L3}) 
\begin{align*} 
	(\na_r\zeta)^\al_i
	&=\sum_{k+j=r}(-1)^k\tbinom{r}{k}D_j\zeta^\al_{i+k}
	=\sum_{k\in\I}(\na_r\eps^k_{\phi_k})^\al_i
	=(-1)^r\sum_{k\in\I}(\eps^{k-r}_{\phi_k})^\al_i \\
	&=(-1)^r\sum_{j\in\I}(\eps^j_{\phi_{j+r}})
	=(-1)^r\sum_{j\in\I}\tbinom{i}{j}D_{i-j}\phi^\al_{i+r}. 
\end{align*}
Thus, we need to satisfy the equality 
\begin{equation*} 
	\sum_{j\in\I}\tbinom{i}{j}D_{i-j}\phi^\al_{i+r}
	=\sum_{k+j=r}(-1)^{k+r}\tbinom{r}{k}D_j\zeta^\al_{i+k}, 
	\quad \al\in\A, \ i\in\I.
\end{equation*}
In particular, for $i=0$ we get 
$\phi^\al_r=\sum_{k+j=r}(-1)^{k+r}\tbinom{r}{k}D_j\zeta^\al_k$. 
The easy test shows that this unique choice solves the problem.
\end{proof}

\begin{cor} 
The following statements hold:
\begin{itemize}
	\item 
		in the filtration $\{\cE^{(q)}\mid q\in\bbZ_+\}$ the linear spaces \\
		$\cE^{(q)}=\oplus_{|k|\le q}\cE^k, \quad 
		\cE^k=\big\{X=\eps^{k\al}_{\phi i}\cdot\p_{u^\al_i} \ 
		\big| \ \eps^k_\phi=(\eps^{k\al}_{\phi i})\in\Phi^k\big \}$; 
	\item 
		$\lim_{q\to\infty}\cE^{(q)}=\fD_V=\oplus_{k\in\I}\cE^k$.  
\end{itemize}
\end{cor} 

Consider the variation bicomplex 
\begin{equation*}
	\{\Om^{pq};d^{pq}_V,d^{pq}_H\mid p\in\bbZ_+,0\le q\le m\}, 
	\quad\text{where}\quad \Om^{pq}=\Om^{pq}(\cA,\cA). 
\end{equation*}
Here, 
\begin{itemize} 
	\item 
		the vertical $\cA$-basis is $\p=\{\p_{u^\al_i}\mid \al\in\A, \ i\in\I\}$ 
		has the dual basis 
		$\rho=\big\{\rho^\al_i=du^\al_i-u^\al_{i+(\mu)}dx^\mu \ \big| \ 
		\al\in\A, \ i\in\I\big\}$, $\rho^\al_i(\p_{u^\be_j})=\de^\al_\be\de^j_i$, 
		$\rho^\al_i(D_\mu)=0$, 
	\item 
		the horizontal $\cA$-basis $D=\{D_\mu\mid \mu\in\bm\}$ 
		has the dual basis $\vte=\{\vte^\mu=dx^\mu\mid\mu\in\bm\}$, 
		$dx^\mu(\p_{u^\al_i})=0$, $dx^\mu(D_\nu)=\de^\mu_\nu$.
\end{itemize} 

We augment the variation bicomplex and add 
\begin{itemize} 
	\item 
		the horizontal complex $\{\Om^q_R;d^q_R\mid 0\le q\le m\}$, 
		where $\Om^q_R=\Om^q(\cC^\infty(X))$, $d^q_R=d^q$, 
		i.e., the standard de Rham complex of the space $X=\bbR^\bm$;
	\item 
	the vertical complex $\{\fF^p;\de^p\mid p\in\bbZ_+\}$, where  
	$\fF^p=\Om^{pm}\big/\im d^{p,m-1}$ are quotient linear spaces of 
	{\it functional $p$-forms}, $\de^p$ are quotient differentials,  
	$[\om]\mapsto\de^p[\om]=[d^{pm}_V\om]$, 
	$[\om]$ is the equivalence class of the form $\om\in\Om^{pm}$.
\end{itemize}

The resulting augmented bicomplex is presented on the page \pageref{abc}.  

\begin{theorem}
The augmented bicomplex is acyclic, i.e., all his raws and columns are exact.
\end{theorem}
\begin{proof} 
The detailed proof  and the history of this famous theorem and close results 
on can find, for example, in \cite{TT},\cite{PO},\cite{AI}.  
\end{proof}

\begin{figure}[t]
\begin{equation*}
{
\begin{diagram}[2em]
&&&&\vdots&&\vdots&&&&\vdots&&\vdots&&  \\
&&&&\uTo^{d^{20}_V}&&\uTo^{d^{21}_V}&&&&\uTo{d^{2m}_V}& 
&\uTo^{\de^2}&&  \\
&&0&\rTo&\Om^{20}&\rTo^{d^{20}_H}&\Om^{21}&\rTo^{d^{21}_H}&\dots
&\rTo^{d^{2,m-1}_H}&\Om^{2m}&\rTo&\fF^2&\rTo&0 \\
&&&&\uTo^{d^{10}_V}&&\uTo^{d^{11}_V}&&&&\uTo{d^{1m}_V}&
&\uTo^{\de^1}&&  \\
&&0&\rTo&\Om^{10}&\rTo^{d^{10}_H}&\Om^{11}&\rTo^{d^{11}_H}&\dots
&\rTo^{d^{1,m-1}_H}&\Om^{1m}&\rTo&\fF^1&\rTo&0 \\
&&\uTo&&\uTo^{d^{00}_V}&&\uTo^{d^{01}_V}&&&&\uTo{d^{0m}_V}&           &\uTo^{\de^0}&&  \\
0&\rTo&\bbF&\rTo&\Om^{00}&\rTo^{d^{00}_H}&\Om^{01}&\rTo^{d^{01}_H}
&\dots&\rTo^{d^{0,m-1}_H}&\Om^{0m}&\rTo&\fF^0&\rTo&0 \\
&&\uTo&&\uTo&&\uTo&&&&\uTo&&\uTo&&  \\
0&\rTo&\bbF&\rTo&\Om_R^0&\rTo^{d^0_R}&\Om^1_R&\rTo^{d^1_R}&\dots&
\rTo^{d^{m-1}_R}&\Om_R^m  &\rTo&0&&  \\
&&\uTo&&\uTo&&\uTo&&&&\uTo&&&&  \\
&&0&&0&&0&&&&0&&&&
\end{diagram}
}
\label{abc}
\end{equation*}
\end{figure} 

In the algebraic approach a nonlinear system of partial differential equations 
is written as $F=0$, where $F=\{F^\si\in\cA\mid \si\in\Sr\}$, 
$\Sr$ is an index set. 
The associated differential ideal  
$\cI_F=\{f=P_\si(\cD) F^\si\mid P_\si(\cD)\in\cA[\cD], \ \si\in\Sr\}$, 
where $\cA[\cD]$ is the unital associative noncommutative algebra 
of all polynomials in indeterminate $\cD=\{D_\mu\mid\mu\in\bm\}$ 
with  coefficients in $\cA$. 
The quotient differential algebra $(\bcA,\bcD)$ is called the 
{\it differential algebra associated to the system $F=0$} 
(see, for example, \cite{Z4},\cite{Z6} for more detail). 
This allows to write the spectral sequence 
({\it the Vinogradov spectral sequence} \cite{AV}) 
associated with the system $F=0$. 
The calculation of this sequence or some of its terms is quite another problem, 
usually extremely hard. 
To construct the associated variation bicomplex one should first contrive 
to write the quotient differential algebra $(\bcA,\bcD)$ in the regular form 
and then follow the procedure presented in Section \ref{S5}. 
Again, one is left with the calculation problem. 

\section{Conclusion.}
The technics and methods presented above were approbated in the author's works 
\cite{Z},\cite{Z0},\cite{Z1},\cite{Z2},\cite{Z4},\cite{Z5},\cite{Z7}. 
They may be useful in the researches 
\cite{KMTV},\cite{MS},\cite{KMO},\cite{SAG},\cite{TV},\cite{GAK},\cite{DYN},\cite{VK}.

\end{document}